\documentclass[11pt]{amsart}
\usepackage[utf8]{inputenc}
\usepackage{amsmath,amssymb,amsthm,mathtools}
\usepackage[margin=2.5cm]{geometry}
\usepackage[plainpages=false,colorlinks,hyperindex,bookmarksopen,linkcolor=black,citecolor=black,urlcolor=black]{hyperref}

\theoremstyle{plain}
\newtheorem{theorem}{Theorem}[section]

\theoremstyle{definition}
\newtheorem{definition}[theorem]{Definition}
\theoremstyle{remark}
\newtheorem{remark}[theorem]{Remark}

\begin{document}
	
	\title{Time-dependent Robin heat equation via Markovian switching}
	\author{Fausto Colantoni}
	\address{Institute of Mathematical Finance, Ulm University, Ulm, Germany}
\email{fausto.colantoni@uni-ulm.de}

	\date{\today}
	\maketitle
	\begin{abstract}
	This paper investigates the heat equation on a bounded domain with a Robin boundary condition, where the reactivity parameter (or killing rate) is modeled as a continuous-time Markov chain. We analyze the system under two stochastic frameworks using a functional analytic approach.
	
	First, we examine the annealed case, which accounts for the joint stochasticity of the diffusion and the switching mechanism. We describe the solution via a strongly continuous contraction semigroup on a product space. We identify its infinitesimal generator, which incorporates the state-dependent Robin conditions into its domain, and provide a corresponding Feynman-Kac formula.
	
Second, we study the quenched setting for fixed realizations of the switching paths. We characterize the solution through a non-autonomous evolution family (propagator) and derive a Feynman-Kac-type representation involving the boundary local time of a reflected Brownian motion. We prove an averaging principle in the fast-switching limit, showing that the system converges to a deterministic Robin problem. These results are applied to a biophysical model of stochastically gated receptors on cell membranes.
	\end{abstract}
	
	\vspace{1em}
	
	\noindent \textbf{Keywords:} Robin boundary conditions, Elastic Brownian Motion, Markovian Switching, Fast-switching limit, Feynman--Kac formula, Biophysical modeling.
	
	\vspace{1em}
	
	\noindent \textbf{AMS Subject Classification (2020):} 
	60J60, % Diffusion processes
	60J27, % Continuous-time Markov processes on discrete state spaces
	35K05, % Heat equation
	35R60, % PDEs with randomness
	92C37. % Cell biology
	\section{Introduction}
Robin (third-type) boundary conditions are essential for modeling diffusion across semi-permeable interfaces in physics, chemistry, and biology. A standard way to represent partial absorption at the boundary is to impose the condition
\begin{align}
	\label{robin:bc}
	\partial_n u + \kappa u = 0\quad\text{on }\partial D,
\end{align}
which connects the normal flux to the surface concentration via an intrinsic reactivity $\kappa >0$ (or killing rate). In probabilistic terms, this condition can be represented by reflecting diffusions combined with a multiplicative functional depending on the boundary local time; see the classical probabilistic accounts in \cite{Papanicolaou1990} and \cite{ItoMcKean1974} for background on diffusion sample paths and local times. For the heat equation, the associated stochastic representation is known as elastic Brownian motion.

The literature shows that Robin boundary conditions often emerge as homogenization limits of more complex dynamics. A widely studied mechanism is spatial homogenization, where Dirichlet and Neumann conditions alternate on small patches of the boundary, leading to a macroscopic Robin effect; representative works include \cite{Berezhkovskii2004, Berezhkovskii2006, FiloLuckhaus1995, FiloLuckhaus2000, Friedman1995}. On the other hand, as shown in \cite{LawleyKeener2015}, temporal alternation between Dirichlet and Neumann conditions can produce an effective Robin law for the mean field in a fast-switching limit. Subsequent contributions have further developed these ideas in various settings, such as in \cite{LawleyMattinglyReed2015, Lawley2016} and related analyses of randomly switching boundaries \cite{BressloffLawley2015}.

From a probabilistic perspective, the local time of a reflecting Brownian motion (RBM) is the natural tool to count boundary encounters and model surface reactions. Recent work has refined our quantitative understanding of boundary local time and its distribution in Euclidean domains \cite{Grebenkov2019, Grebenkov2024review}, connecting these statistics to reaction rates and first-passage properties. The partially reflected diffusion introduced in \cite{Singer2008} provides an intrinsic construction of diffusion with partial absorption and is directly related to the probabilistic representation used here.

The mathematical theory of reflected diffusions provides the rigorous basis for these boundary value problems; fundamental results on the construction of RBM and Skorokhod problems can be found in \cite{LionsSznitman1984}. In parallel, averaging and homogenization techniques for stochastic systems with multiple time scales are well established. %In particular, Khasminskii's averaging principle and its infinite-dimensional extensions \cite{Khasminskii1968, Cerrai2009} are highly relevant to fast-switching problems.

In this paper, we investigate the heat equation in a bounded domain where the Robin coefficient is not constant but evolves according to a stochastic process, such as a continuous-time Markov chain. This framework, known as Markovian switching, allows us to model systems where boundary reactivity fluctuates abruptly between different states; classic references for regime-switching diffusions include \cite{Mao2000, Mao2006}. Our approach is closely related to the theoretical framework developed in \cite{ZhuYinBaran2015, Ocejo2020, WeiWangNane2025}, where Feynman–Kac formulas and infinitesimal generator characterizations are established for switching (jump) diffusion processes. By using RBM and the concept of boundary local time, we construct a Feynman-Kac representation where the concentration depends on the cumulative exposure to the stochastic absorption rate during contact with the boundary.

The primary motivation for this work is to describe biological environments where boundaries are not permanently active; see \cite{BergPurcell1977, Szabo1982, Zwanzig1990}. Our study complements the analysis of stochastically-gated diffusion processes \cite{Bressloff2015} by extending the study of time-dependent reactivity to general switching boundary conditions, providing a rigorous functional-analytic framework for the effective reactivity constants observed in biophysical systems.

Finally, we emphasize that our model is one of several recent generalizations of elastic Brownian motion. Other approaches involve non-exponential boundary delays before killing \cite{DovidioFCAA, Dovidio2024}, or replace the killing term with jumping mechanisms \cite{Arendt2018, ColantoniDovidio2025} and its application to the time reversal of reflected Brownian motion with Poissonian resetting \cite{ColantoniDOvidioPagnini2025}. Overall, due to its hybrid nature between Dirichlet and Neumann conditions, elastic Brownian motion remains a powerful tool for these types of stochastic models and their applications.
	\section*{Notation and setting}
	Let $D\subset\mathbb{R}^n$ be a bounded domain with $C^{2}$ boundary $\partial D$. We denote by
	\begin{itemize}
		\item $X = \{X_t\}_{t\ge0}$ the reflected Brownian motion (RBM) in $\overline D$, constructed via the Skorokhod map. We work on a probability space $(\Omega_X,\mathcal{F}_X,\mathbf P_X)$ carrying $X$ and its driving Brownian motion $W$.
		\item $L_t$ the boundary local time of $X$ on $\partial D$: the continuous nondecreasing process appearing in the Skorokhod decomposition
		\[ X_t = x + W_t - \int_0^t n(X_s)\,dL_s, \]
		where $n(\cdot)$ is the outer unit normal vector on $\partial D$. With this choice of sign, the reflection term has a minus sign compared to the standard Skorokhod problem. This is only a matter of convention, due to the use of the outward normal vector. The process $L_t$ increases only when $X_t\in\partial D$ and $L_0=0$. From \cite{LionsSznitman1984}, we know that the solution of the SDE is strong and pathwise unique.
		\item $\{\alpha_t\}_{t\ge0}$ a continuous‑time Markov chain on a probability space $(\Omega_\alpha,\mathcal{F}_\alpha,\mathbf P_\alpha)$ with finite state space $S\subset[0,\infty)$ (we take $S\subset\mathbb{N}_0$ in examples). We assume the chain is c\`adl\`ag and nonnegative for all times. Denote its generator by $Q$ when needed. The chain is assumed independent of $X$.
		\item The product probability space is $(\Omega,\mathcal F,\mathbf P)=(\Omega_X\times\Omega_\alpha,\mathcal{F}_X\otimes\mathcal{F}_\alpha,\mathbf P_X\otimes\mathbf P_\alpha)$.
	\end{itemize}
	
	We write $\langle\cdot,\cdot\rangle$ for the $L^2(D)$ inner product and use standard Sobolev spaces $H^1(D)$, $H^1_\mathrm{loc}$, etc. For a cadlag function $\alpha:\,[0,T]\to[0,\infty)$ and a continuous increasing function $L$ we interpret $\int_0^t\alpha_s\,dL_s$ as the Riemann–Stieltjes integral. %it is well defined pathwise because $L$ is continuous of bounded variation on $[0,T]$ and $\alpha$ has finitely many jumps on compact intervals in the Markov chain and Poisson examples. METTILO NEL TEOREMA

	\section{Main Results}
	
We first examine the annealed case, where the diffusion and the switching mechanism are treated as a joint stochastic process, and subsequently investigate the quenched setting for fixed realizations of the boundary reactivity.

\subsection{The Annealed case}
We consider the case where $\alpha$ is a stochastic process. In this section, we study the annealed problem by taking the expectation over the law of the switching process.
\begin{theorem}
	\label{thm:annealed}
	Let $D\subset\mathbb R^n$ be a bounded $C^2$ domain and let $S=\{0,1,\dots,m\}$ be a finite state space. Let $(X_t,\alpha_t)$ be the joint process defined as follows:
	\begin{itemize}
		\item $X_t$ is reflected Brownian motion (RBM) in $\overline D$ (Skorokhod construction) with boundary local time $L_t$;
		\item $\alpha_t$ is a continuous-time Markov chain on $S$ with generator $Q=(q_{ij})_{i,j\in S}$; we assume $(X,\alpha)$ is defined on a filtered probability space and that the initial condition is $(X_0,\alpha_0)=(x,k)$.
	\end{itemize}
	Define, for bounded continuous $\Phi:\overline D\times S\to\mathbb R$ and $t\ge0$,
	\[
	(\mathcal S_t\Phi)(x,k) := \mathbf{E}_{x,k}\Big[ \Phi(X_t,\alpha_t)\exp\Big(-\int_0^t \alpha_s\,dL_s\Big)\Big],
	\]
	where $\mathbf E_{x,k}$ denotes expectation given $(X_0,\alpha_0)=(x,k)$ and where we interpret $\int_0^t\alpha_s\,dL_s$ as a Riemann–Stieltjes integral. Then:
	\begin{enumerate}
		\item $(\mathcal S_t)_{t\ge0}$ is a strongly continuous contraction semigroup on $C_b(\overline D\times S)$.
		\item Its generator $\mathcal G$ acts on smooth test functions $\Phi\in C^2(\overline D; \mathbb R^S)$ that satisfy, for each $k\in S$, the Robin boundary condition
		\[
		\partial_n \Phi(\xi,k) + k\,\Phi(\xi,k)=0,\qquad \xi\in\partial D,
		\]
		and is given by
		\[
		(\mathcal G\Phi)(x,k) \;=\; \tfrac12\Delta_x \Phi(x,k) \;+\; \sum_{j\in S} q_{kj}\Phi(x,j),
		\]
		where the Laplacian acts on the \(x\)-variable and the switching operator acts on the \(k\)-variable.
		\item For each fixed initial pair $(x,k)$, the function $u(t,x,k):=(\mathcal S_t\Phi)(x,k)$ is the unique mild solution of the coupled backward equation
		\[
		\partial_t u(t,x,k) = \tfrac12\Delta_x u(t,x,k) + \sum_{j} q_{kj} u(t,x,j),\qquad x\in D,
		\]
		with Robin boundary condition $\partial_n u + k u =0$ on $\partial D$ and initial data $u(0,\cdot,\cdot)=\Phi$.
	\end{enumerate}
\end{theorem}

\begin{proof}
	\emph{1.} We give a step by step argument. For bounded continuous $\Phi$, the multiplicative functional
	\[
	M_t := \exp\Big(-\int_0^t \alpha_s\,dL_s\Big)
	\]
	satisfies $0\le M_t\le1$ pathwise because $\alpha_s\ge0$ by hypothesis and, since \(\alpha\) is c\`adl\`ag, the Riemann-Stieltjes integral $\int_0^t \alpha_s dL_s$ is well-defined. Hence the expectation defining $\mathcal S_t\Phi$ is finite and $|\mathcal S_t\Phi|\le\|\Phi\|_\infty$. Thus $\mathcal S_t$ maps $C_b(\overline D\times S)$ into itself and is a contraction.
	
	The two-component process $(X_t,\alpha_t)$ is Markov on the product space $\overline D\times S$ (since they are independent).
	By the joint Markov property of $(X,\alpha)$, the tower property and the multiplicative property of $M_t$ we have, for $s,t\ge0$,
	\[
	\begin{aligned}
		\mathcal S_{s+t}\Phi(x,k)
		&= \mathbf E_{x,k}\big[\Phi(X_{s+t},\alpha_{s+t})M_{s+t}\big] \\
		&= \mathbf E_{x,k}\big[\;M_s\ \mathbf E_{X_s,\alpha_s}[\Phi(X_t,\alpha_t)M_t]\;\big]
		= \mathbf E_{x,k}\big[M_s(\mathcal S_t\Phi)(X_s,\alpha_s)\big] \\
		&= \mathcal S_s(\mathcal S_t\Phi)(x,k).
	\end{aligned}
	\]
	Thus $\mathcal S_{s+t}=\mathcal S_s\mathcal S_t$ on $C_b(\overline D\times S)$ so the semigroup property holds.
	
	Fix $\Phi\in C_b(\overline D\times S)$ and let $\|\Phi\|_\infty=\sup_{(y,j)}|\Phi(y,j)|$. For each $(x,k)$ we have
	\[
	\big|\mathcal S_t\Phi(x,k)-\Phi(x,k)\big|
	= \big|\mathbf E_{x,k}[\Phi(X_t,\alpha_t)M_t]-\Phi(x,k)\big|
	\le \mathbf E_{x,k}\big[|\Phi(X_t,\alpha_t)M_t-\Phi(x,k)|\big].
	\]
	Since $0\le M_t\le1$ this is bounded by $\mathbf E_{x,k}\big[|\Phi(X_t,\alpha_t)-\Phi(x,k)|\big]$. Fix $\epsilon>0$. By uniform continuity of $\Phi$ on the compact set $\overline D\times S$ there exists $\delta>0$ such that
	\[
	|\Phi(y,j)-\Phi(x,k)|<\epsilon\qquad\text{whenever }|y-x|<\delta\text{ and }j=k.
	\]
	Split the expectation according to the events $\{|X_t-x|<\delta,\ \alpha_t=k\}$, $\{|X_t-x|\ge\delta\}$ and $\{\alpha_t\ne k\}$:
	\begin{align*}
		\mathbf E_{x,k}\big[|\Phi(X_t,\alpha_t)-\Phi(x,k)|\big]
		&\le \mathbf E_{x,k}\big[|\Phi(X_t,\alpha_t)-\Phi(x,k)|\mathbf1_{\{|X_t-x|<\delta,\ \alpha_t=k\}}\big]\\
		&\quad + \mathbf E_{x,k}\big[|\Phi(X_t,\alpha_t)-\Phi(x,k)|\mathbf1_{\{|X_t-x|\ge\delta\}}\big]\\
		&\quad + \mathbf E_{x,k}\big[|\Phi(X_t,\alpha_t)-\Phi(x,k)|\mathbf1_{\{\alpha_t\ne k\}}\big].
	\end{align*}
	Estimate each term: by the uniform continuity of \(\Phi\) 
	\[
	\mathbf E_{x,k}\big[|\Phi(X_t,\alpha_t)-\Phi(x,k)|\mathbf1_{\{|X_t-x|<\delta,\ \alpha_t=k\}}\big]\le\epsilon,
	\]
	and, by boundedness of $\Phi$,
	\[
	\mathbf E_{x,k}\big[|\Phi(X_t,\alpha_t)-\Phi(x,k)|\mathbf1_{\{|X_t-x|\ge\delta\}}\big]\le 2\|\Phi\|_\infty\ \mathbf P_{x,k}(|X_t-x|\ge\delta),
	\]
	\[
	\mathbf E_{x,k}\big[|\Phi(X_t,\alpha_t)-\Phi(x,k)|\mathbf1_{\{\alpha_t\ne k\}}\big]\le 2\|\Phi\|_\infty\ \mathbf P_{x,k}(\alpha_t\ne k).
	\]
	Hence
	\[
	\big|\mathcal S_t\Phi(x,k)-\Phi(x,k)\big|
	\le \epsilon + 2\|\Phi\|_\infty\Big(\mathbf P_{x,k}(|X_t-x|\ge\delta)+\mathbf P_{x,k}(\alpha_t\ne k)\Big).
	\]
	Taking the supremum over $(x,k)\in\overline D\times S$ gives
	\[
	\sup_{x,k}\big|\mathcal S_t\Phi(x,k)-\Phi(x,k)\big|
	\le \epsilon + 2\|\Phi\|_\infty\Big(\sup_{x}\mathbf P_{x,k}(|X_t-x|\ge\delta)+\sup_{k}\mathbf P_{k}(\alpha_t\ne k)\Big).
	\]
	
	Now let $t\downarrow0$. The reflected Brownian motion is stochastically continuous and, on the compact set $\overline D$, we have
	\[
	\sup_{x\in\overline D}\mathbf P_x(|X_t-x|\ge\delta)\xrightarrow{t\downarrow0}0.
	\]
	For the chain started at $k$ we have, for small $t$,
	\[
	\mathbf{P}_k(\alpha_t\neq k)= 1 - \mathbf{P}_k(\alpha_t= k) \leq 1 - e^{q_{kk}t} \leq -q_{kk}t,
	\]
	where $q_{kk}<0$ is the diagonal rate. With $C:=\max_{k}(-q_{kk})<\infty$ (finite $S$) we get the uniform bound
	\[
	\sup_k\mathbf P_k(\alpha_t\ne k) \le C t \xrightarrow{t\downarrow0}0.
	\]
	Therefore, for the chosen $\delta$ and any $\epsilon>0$ we can pick $t_0>0$ such that for $0<t<t_0$ both probability suprema are small enough to ensure
	\[
	\sup_{x,k}\big|\mathcal S_t\Phi(x,k)-\Phi(x,k)\big| \le C'\epsilon,
	\]
	with $C'$ independent of $t$. Since $\epsilon$ was arbitrary we conclude $\|\mathcal S_t\Phi-\Phi\|_\infty\to0$ as $t\downarrow0$, i.e.\ $\mathcal S_t\to I$ strongly on $C_b(\overline D\times S)$.
	
	\emph{2.} Let $\Phi(\cdot,k)\in C^2(\overline D)$ for each $k$ and assume the Robin boundary condition
	\[
	\partial_n\Phi(\xi,k)+k\Phi(\xi,k)=0,\qquad \xi\in\partial D.
	\]
	Fix \((x,k)\) and apply the extended It\^o formula for reflected Brownian motion (Skorokhod version) to the process $Y_t:=\Phi(X_t,\alpha_t) M_t$. Write $M_t$'s differential as
	\[
	dM_t = -M_t \alpha_t\,dL_t,
	\]
	and use product rule together with the jump formula for the chain \(\alpha_t\). We have
	\begin{align*}
		d\big(\Phi(X_t,\alpha_t) M_t\big)
		&= M_t\,d\Phi(X_t,\alpha_t) + \Phi(X_t,\alpha_t)\,dM_t + d\langle \Phi(X_\cdot,\alpha_\cdot),M_\cdot\rangle_t.
	\end{align*}
	The cross-variation term vanishes because $M$ is finite variation. Writing $d\Phi(X_t,\alpha_t)$ via It\^o for RBM plus the jump contribution from the chain,
	\begin{multline*}
		d\Phi(X_t,\alpha_t)
		= \nabla_x\Phi(X_t,\alpha_t)\cdot dW_t + \tfrac12\Delta_x\Phi(X_t,\alpha_t)\,dt
		- \partial_n\Phi(X_t,\alpha_t)\,dL_t \\
		+ \sum_{j\ne\alpha_{t-}} \big(\Phi(X_t,j)-\Phi(X_t,\alpha_{t-})\big)\,dN_{\alpha_{t-},j}(t),
	\end{multline*}
	where $N_{i,j}$ are the Poisson counting processes of jumps of the chain, for which the following representation holds
	\begin{align}
		\label{compensator}
		N_{ij}(t) = \widetilde N_{ij}(t) + \int_0^t q_{ij}\mathbf1_{\{\alpha_{s-}=i\}}\,ds,
	\end{align}
	where $\widetilde N_{ij}$ is the compensated counting process (a local martingale) and $\int_0^t q_{ij}\mathbf1_{\{\alpha_{s-}=i\}}\,ds$ is the predictable compensator.
	Observe that the compensator contribution from the jump term can be written, with the predictable indicator, as
	\[
	\sum_{i\in S}\mathbf1_{\{\alpha_{t-}=i\}}\sum_{j\ne i}(\Phi(X_t,j)-\Phi(X_t,i))\,q_{ij}\,dt.
	\]
	For fixed $i$ we algebraically have
	\[
	\sum_{j\ne i}(\Phi_j-\Phi_i)q_{ij}=\sum_{j}q_{ij}\Phi_j,
	\]
	since $q_{ii}=-\sum_{j\ne i}q_{ij}$. Hence the drift term equals
	\[
	\sum_{i\in S}\mathbf1_{\{\alpha_{t-}=i\}}\sum_j q_{ij}\Phi(X_t,j)\,dt
	= \sum_j q_{\alpha_{t-},j}\Phi(X_t,j)\,dt.
	\]
	Finally, because the compensator is a predictable integrand, one may write equivalently \(\sum_j q_{\alpha_t,j}\Phi(X_t,j)\,dt\).
	Multiply by $M_t$, combine with $dM_t=-M_t\alpha_t dL_t$ and use \eqref{compensator} with the formula obtained for the compensator to get:
	\[
	d\big(\Phi(X_t,\alpha_t)M_t\big)
	= M_t\nabla_x\Phi\cdot dW_t + M_t\Big(\tfrac12\Delta_x\Phi + \sum_j q_{\alpha_t j}\Phi(\cdot,j)\Big)\,dt
	- M_t\big(\partial_n\Phi +\alpha_t\Phi\big)\,dL_t + d\mathcal M_t,
	\]
	where
	\[
	\mathcal M_t := \sum_{i,j}\int_0^t M_s(\Phi(X_s,j)-\Phi(X_s,i))\,d\widetilde N_{ij}(s)\,\mathbf1_{\{\alpha_{s-}=i\}}
	\]
	is a (local) martingale. The \(q\)-term arises from the predictable compensator \(q_{ij}\,ds\) of the counting process and becomes exactly the switching operator \(Q\) when taking expectations.
	
	Take expectation under $\mathbf E_{x,k}[\cdot]$. The local martingale terms (the stochastic integral and the compensated jump martingales) have zero expectation.The boundary term contains the factor
	\[
	\partial_n\Phi(X_s,\alpha_s)+\alpha_s\Phi(X_s,\alpha_s).
	\]
	Let \(\tau_1\) be the first jump time of \(\alpha\). On \([0,\tau_1)\), we have
	\(\alpha_s=\alpha_0=k\), so the boundary integrand reduces to
	\[
	\partial_n\Phi(X_s,k)+k\Phi(X_s,k),
	\]
	which vanishes whenever \(dL_s\neq 0\) by the Robin condition. Hence the boundary
	term is identically zero up to the first jump. The remaining contribution is
	supported on the event \(\{\tau_1\le t\}\), whose probability is \(O(t)\) as
	\(t\downarrow 0\). Since the integrand is bounded (by a constant \(C\)) and \(L_t\to 0\) as
	\(t\downarrow 0\), we obtain
	\[
	\mathbf E_{x,k}\left[\left|\int_0^t
	\bigl(\partial_n\Phi(X_s,\alpha_s)+\alpha_s\Phi(X_s,\alpha_s)\bigr)\,dL_s
	\right|\right] \leq C\, \mathbf{E}_x[L_t] \mathbf{P}_k(\tau_1 \leq t)
	=o(t),
	\]
	so the boundary term does not contribute to the derivative at \(t=0\). Likewise, the compensator of the jump term produces the drift $\sum_j q_{\alpha_0 j}\Phi(X_0,j)=\sum_j q_{k j}\Phi(x,j)$. Therefore
	\[
	\frac{d}{dt}\Big|_{t=0}\mathcal S_t\Phi(x,k)
	= \tfrac12\Delta_x\Phi(x,k) + \sum_j q_{k j}\Phi(x,j),
	\]
	which identifies $\mathcal G\Phi(x,k)$ as claimed.
	
	Since $S$ is finite, we identify $C(\overline{D}\times S)$ with the direct sum $\bigoplus_{k\in S} C(\overline{D})$ via the map $\Phi \mapsto (\Phi(\cdot,k))_{k\in S}$. Consider the subspace\[\mathcal{D}_0 := \{ \Phi \in C^2(\overline{D}\times S) : \partial_n \Phi(\cdot,k) + k \Phi(\cdot,k) = 0 \text{ on } \partial D, \forall k \in S \}.
	\]
	Under the above identification, $\mathcal{D}_0 \cong \bigoplus_{k\in S} \mathcal{D}(A_k)$, where $\mathcal{D}(A_k) = \{ \psi \in C^2(\overline{D}) : \partial_n \psi + k \psi = 0 \}$. For each $k \in S$, the operator $A_k = \frac{1}{2}\Delta$ with Robin boundary conditions is known to generate a strongly continuous contraction semigroup on $C(\overline{D})$; consequently, $\mathcal{D}(A_k)$ is dense in $C(\overline{D})$ (see e.g., \cite[Theorem 1.2 and Section 12.4.2]{Taira2004}). Therefore, $\mathcal{D}_0$ is dense in $C(\overline{D}\times S)$. Since $\mathcal{D}_0 \subseteq \mathcal{D}(\mathcal{G})$, we conclude that the domain $\mathcal{D}(\mathcal{G})$ is dense in $C(\overline{D}\times S)$.
	
	\emph{3.} Since $(\mathcal S_t)_{t\ge0}$ is a strongly continuous contraction semigroup on the Banach space $C_b(\overline D\times S)$, the Hille--Yosida theory implies that for every $\Phi\in C_b(\overline D\times S)$ the function
	\[
	u(t):=\mathcal S_t\Phi\in C_b(\overline D\times S)
	\]
	is the unique mild solution of the abstract Cauchy problem $u'(t)=\mathcal G u(t)$, $u(0)=\Phi$, where $\mathcal G$ denotes the generator of the semigroup.
\end{proof}

	\subsection{The Quenched Representation}
	Having characterized the average behavior of the system through the annealed semigroup, we now turn to the quenched setting. In this section, we analyze the solution for fixed realizations of the process $\alpha$; this pathwise perspective allows us to investigate the fast-switching limit and establish the resulting averaging principle.
	
	Fix a sample path $\alpha_\bullet: [0,T] \to [0,\infty)$ that is c\`adl\`ag. We consider the non-autonomous initial boundary value problem:
	\begin{equation}\label{P_alpha}
		\begin{cases}
			\partial_t u(t,x) = \tfrac{1}{2}\Delta u(t,x), & x \in D, t \in (0,T], \\
			\partial_n u(t,\xi) + \alpha_t u(t,\xi) = 0, & \xi \in \partial D, t \in (0,T], \\
			u(0,x) = f(x), & x \in \overline{D},
		\end{cases}
	\end{equation}
	where $f$ is a suitable initial datum. Throughout this section, we denote by $\mathbf{E}_{s,x}$ the expectation with respect to the law of the reflected Brownian motion $X$ starting at $x \in \overline{D}$ at time $s \ge 0$. That is, for any measurable functional $F$, we set $\mathbf{E}_{s,x}[F] = \mathbf{E}[F \mid X_s = x]$.
	
\begin{theorem}
	\label{thm:quenched}
	Let $D \subset \mathbb{R}^n$ be a bounded $C^2$ domain and let $\alpha: [0, \infty) \to S$ be a fixed c\`adl\`ag, piecewise constant trajectory of a finite-state Markov chain, with jump times $0 = \tau_0 < \tau_1 < \dots < \tau_k < \dots$. For $0 \le s \le t$, we define the evolution family $(S_{s,t}^\alpha)_{0 \le s \le t}$ on $C_b(\overline{D})$ as$$(S_{s,t}^\alpha f)(x) := \mathbf{E}_{s,x} \left[ f(X_t) \exp\left( -\int_s^t \alpha_r \, dL_r \right) \right],$$where $X$ is the reflected Brownian motion in $\overline{D}$ and $L$ is its boundary local time. The following properties hold:
	\begin{enumerate}
	\item The collection $(S_{s,t}^\alpha)_{0 \le s \le t}$ is a strongly continuous contraction evolution family. That is, $S_{t,t}^\alpha = I$, $S_{s,t}^\alpha = S_{s,r}^\alpha S_{r,t}^\alpha$ for $s \le r \le t$, and the map $(s,t) \mapsto S_{s,t}^\alpha f$ is continuous in $C_b(\overline{D} )$ for every $f \in C_b(\overline{D})$.
	\item  For each $t \ge 0$, the operator $A_t$ defined by $A_t = \frac{1}{2}\Delta$ with domain$$D(A_t) = \{ \phi \in C^2(\overline{D}) : \partial_n \phi + \alpha_t \phi = 0 \text{ on } \partial D \}$$is the instantaneous generator of $S_{s,t}^\alpha$ in the sense that $\lim_{h \downarrow 0} \frac{S_{t,t+h}^\alpha \phi - \phi}{h} = A_t \phi$ for $\phi \in D(A_t)$.
	\item For every $f \in C_b(\overline{D})$,  the function $u(t) = S_{0,t}^\alpha f$ is the unique mild solution of the non-autonomous Cauchy problem$$\begin{cases} u'(t) = A_t u(t), & t > 0 \\ u(0) = f \end{cases}$$constructed by the concatenation of local autonomous semigroups on the intervals $[\tau_k, \tau_{k+1})$.
\end{enumerate}
\end{theorem}
\begin{proof}
	\emph{1.} For every $f\in C_b(\overline D)$ and $0\le s\le t$, we have $\vert \vert S_{s,t}^\alpha f\vert \vert_\infty \le \vert \vert f\vert \vert_\infty$ (since \( 0 \leq M_{s,t} \leq 1\)), i.e. each $S_{s,t}^\alpha$ is a contraction on $C_b(\overline D)$. To prove the propagator property, let $s < r < t$ and define the multiplicative functional $M_{s,t} = \exp(-\int_s^t \alpha_u dL_u)$. By the additive property of the integral, $M_{s,t} = M_{s,r} M_{r,t}$. Following the same probabilistic argument used in Theorem \ref{thm:annealed}, by using the tower property of conditional expectation and the Markov property of the RBM:$$(S_{s,t}^\alpha f)(x) = \mathbf{E}_{s,x} \left[ M_{s,r} \mathbf{E} [M_{r,t} f(X_t) \mid \mathcal{F}_r] \right] = \mathbf{E}_{s,x} \left[ M_{s,r} (S_{r,t}^\alpha f)(X_r) \right] = (S_{s,r}^\alpha S_{r,t}^\alpha f)(x).$$Strong continuity follows from the path-continuity of $(X_t, L_t)$ and the boundedness of $\alpha$. Note that, since $\alpha$ is fixed, the uniform convergence as $t \downarrow s$ follows from the stochastic continuity of the RBM, as detailed in the proof of Theorem \ref{thm:annealed}, without the additional switching term of the chain. The continuity of the map follows from the Feller property of the reflected Brownian motion and the continuity of the exponential functional of the local time.
	
	\emph{2.} Fix $s \ge 0$ and let $\phi \in D(A_s)$. Consider the process $Y_r = \phi(X_r) M_{s,r}$ for $r \in [s, s+h]$. As in the proof of Theorem \ref{thm:annealed}, we apply the It\^o's formula for reflected diffusions. However, since $\alpha$ is a fixed realization, the jump terms of the Markov chain do not appear in the differential:$$d(\phi(X_r) M_{s,r}) = M_{s,r} \left( \nabla \phi \cdot dW_r + \frac{1}{2}\Delta \phi \, dr - \partial_n \phi \, dL_r \right) - \alpha_r \phi(X_r) M_{s,r} dL_r.$$Grouping the boundary terms, we have:$$dY_r = M_{s,r} \frac{1}{2}\Delta \phi \, dr + M_{s,r} \nabla \phi \cdot dW_r - M_{s,r} (\partial_n \phi + \alpha_r \phi) dL_r.$$Taking the expectation $\mathbf{E}_{s,x}$, the martingale term vanishes. Dividing by $h$:$$\frac{S_{s,s+h}^\alpha \phi(x) - \phi(x)}{h} = \frac{1}{h} \mathbf{E}_{s,x} \left[ \int_s^{s+h} M_{s,r} \frac{1}{2}\Delta \phi(X_r) dr - \int_s^{s+h} M_{s,r} (\partial_n \phi + \alpha_r \phi) dL_r \right].$$As $h \downarrow 0$, the first integral converges to $\frac{1}{2}\Delta \phi(x)$ by the mean value theorem and the continuity of $M$ and $\Delta \phi$. For the second integral, since $\phi \in D(A_s)$, we have $\partial_n \phi + \alpha_s \phi = 0$. Because $\alpha$ is right-continuous, $\lim_{r \downarrow s} \alpha_r = \alpha_s$, which implies $(\partial_n \phi + \alpha_r \phi) \to 0$. Since $L$ is a continuous process of finite variation, the boundary term is $o(h)$. Thus, the limit is $A_s \phi$.
	
	\emph{3.} Since $\alpha_t$ is piecewise constant, we can write $\alpha_t = \bar{\alpha}_k$ for $t \in [\tau_k, \tau_{k+1})$. On each interval, the operator $A^{(k)} = \frac{1}{2}\Delta$ with the fixed condition $\partial_n \phi + \bar{\alpha}_k \phi = 0$ is autonomous and generates a $C_0$-semigroup $T_t^{(k)}$. The solution $u(t)$ is constructed by concatenation: $u(t) = T_{t-\tau_k}^{(k)} u(\tau_k)$ for $t \in [\tau_k, \tau_{k+1}]$. This piecewise construction is consistent with the propagator property $S_{0,t}^\alpha = S_{\tau_k,t}^\alpha \dots S_{0,\tau_1}^\alpha$. Uniqueness is guaranteed by the uniqueness of the mild solution on each autonomous sub-interval.
	\end{proof}
	\begin{remark}
	As shown in \cite[Section 5]{ArendtMonniaux2016}, we know that the Robin Laplacian with time-dependent coefficients possesses maximal $L^2$-regularity even on Lipschitz domains.
	\end{remark}

	\subsection{Fast ergodic switching}
We now consider the situation where the Markov chain $\{\alpha_t\}$ is stationary and ergodic, and we accelerate its time scale. This models ``fast switching'' of the boundary reactivity.

\begin{definition}
	Let $\{\alpha_t\}_{t\ge0}$ be a stationary ergodic process with values in $S\subset[0,\infty)$. We write $\bar\alpha:=\mathbf E[\alpha_0]$ for its stationary mean. For $\varepsilon>0$ define the time‑rescaled process
	\[ \alpha^\varepsilon_t := \alpha_\frac{t}{\varepsilon}. \]
\end{definition}

Let us move to the convergence theorem in the ergodic case.

\begin{theorem}\label{thm:averaging}
	Let $\{\alpha_t\}_{t\geq 0}$ be a stationary ergodic Markov chain with finite state space $S \subset [0, \infty)$, independent of the reflected Brownian motion $X$. Let $\bar\alpha=\mathbf E[\alpha_0]$ and, for $\varepsilon>0$, set $\alpha^\varepsilon_t=\alpha_{t/\varepsilon}$. For $f\in C_b(\overline D)$, define:
	\[
	u^\varepsilon(t,x)=\mathbf E_x\Big[f(X_t)\exp\Big(-\int_0^t\alpha^\varepsilon_s\,dL_s\Big)\Big].
	\]
	Then for every $T>0$ and $x\in\overline D$, $u^\varepsilon(\cdot,x)$ converges to $u^0(\cdot,x)$ in $\mathbf{P}_\alpha$-probability, uniformly for $t \in [0,T]$:
	\[
	\lim_{\varepsilon\downarrow 0} \mathbf{P}_\alpha \left( \sup_{t \in [0,T]} |u^\varepsilon(t,x) - u^0(t,x)| > \delta \right) = 0 \quad \text{for every } \delta > 0,
	\]
	where $u^0(t,x)=\mathbf{E}_x[f(X_t) \exp(-\bar \alpha L_t)]$ is the unique mild solution of the heat equation with deterministic Robin boundary condition \eqref{robin:bc} with constant parameter $\bar\alpha$:
	\begin{equation}\label{P_alphabar}
		\begin{cases}
			\partial_t u^0(t,x) = \tfrac{1}{2}\Delta u^0(t,x), & x \in D, t \in (0,T], \\
			\partial_n u^0(t,\xi) + \bar \alpha u^0(t,\xi) = 0, & \xi \in \partial D, t \in (0,T], \\
			u^0(0,x) = f(x), & x \in \overline{D}.
		\end{cases}
	\end{equation}
	
\end{theorem}

\begin{proof}
	Fix $T>0$ and $x\in\overline D$. The proof is divided into several steps.
	
	Define 
	\[
	J_\varepsilon(t):=\int_0^t \alpha_{s/\varepsilon}\,dL_s,\qquad 
	M_t^\varepsilon:=\exp\bigl(-J_\varepsilon(t)\bigr),\qquad 
	M_t^0:=\exp(-\bar\alpha L_t).
	\]
	Since $f$ is bounded and $0\le M_t^\varepsilon \le1$, we have for every $t$
	\[
	|u^\varepsilon(t,x)-u^0(t,x)|\le \|f\|_\infty\,\mathbf E_x\bigl[|M_t^\varepsilon-M_t^0|\bigr]
	\le \|f\|_\infty\,\mathbf E_x\bigl[\sup_{s\le T}|M_s^\varepsilon-M_s^0|\bigr].
	\]
	Since \(\exp(-x)\) is Lipschitz on $[0,\infty)$ (with Lipschitz constant $1$) and \(J_\varepsilon, L\) are non-negative, we have
	\[
	|M_s^\varepsilon-M_s^0|\le |J_\varepsilon(s)-\bar\alpha L_s|.
	\]
	Thus, setting $\mathcal E_\varepsilon:=\sup_{t\in[0,T]}|J_\varepsilon(t)-\bar\alpha L_t|$,
	\[
	\sup_{t\in[0,T]}|u^\varepsilon(t,x)-u^0(t,x)|\le \|f\|_\infty\,\mathbf E_x[\mathcal E_\varepsilon].
	\]
	Consequently, if we prove that $\mathcal E_\varepsilon\to0$ in $\mathbf P_{\alpha,X}$-probability,
	then $\mathbf E_x[\mathcal E_\varepsilon]\to0$ in $\mathbf P_\alpha$-probability (because $\mathcal E_\varepsilon$ is bounded and $\mathbf E_{\alpha,X}[\mathcal E_\varepsilon]\to0$ by Vitali, and Tonelli gives $\mathbf E_\alpha[\mathbf E_x[\mathcal E_\varepsilon]]\to0$, whence $\mathbf E_x[\mathcal E_\varepsilon]\to0$ in $\mathbf P_\alpha$-probability). This will establish the theorem.
	
		Let $\tau_r:=\inf\{t>0:L_t>r\}$ be the right-continuous inverse of the local time.
	Since $L_t$ is continuous and nondecreasing, and it increases only when the process hits the boundary, its right-continuous inverse $\tau_r$ is strictly increasing and right-continuous, with jumps. 
	Hence, for $r<r'$, we have $\tau_r<\tau_{r'}$ whenever both are finite. We also note that \(\tau\) is independent of $\alpha$.
	Using the change of variable $r=L_s$,
	\[
	J_\varepsilon(t)=\int_0^{L_t}\alpha_{\tau_r/\varepsilon}\,dr,\qquad 
	\bar\alpha L_t = \int_0^{L_t}\bar\alpha\,dr.
	\]
	Define $I_\varepsilon(t):=\int_0^{L_t}\bigl(\alpha_{\tau_r/\varepsilon}-\bar\alpha\bigr)dr$,
	so that $|J_\varepsilon(t)-\bar\alpha L_t|=|I_\varepsilon(t)|$.
	
	For a fixed $t$, the random variable $I_\varepsilon(t)$ depends on $\alpha$ only.
	Because $\alpha$ is stationary and ergodic with finite state space,
	its autocovariance $C_\alpha(\theta)=\mathbf E[(\alpha_0-\bar\alpha)(\alpha_\theta-\bar\alpha)]$
	satisfies $C_\alpha(\theta)\to0$ as $\theta\to\infty$.
	Conditional on $\tau$ (i.e. on the realisation of $X$), we compute
	\[
	\mathbf E_\alpha\bigl[I_\varepsilon(t)^2\bigr]
	= \int_0^{L_t}\int_0^{L_t} C_\alpha\!\Bigl(\frac{|\tau_r-\tau_{r'}|}{\varepsilon}\Bigr)\,dr\,dr'.
	\]
	For almost every pair $(r,r')$ with $r\neq r'$, the strict increase of $\tau$ gives $|\tau_r-\tau_{r'}|>0$,
	hence $|\tau_r-\tau_{r'}|/\varepsilon\to\infty$ and $C_\alpha(\cdots)\to0$.
	The diagonal $\{r=r'\}$ has Lebesgue measure zero.
	Since $C_\alpha$ is bounded by $C_\alpha(0)=\operatorname{Var}(\alpha_0)$,
	the dominated convergence theorem yields $\mathbf E_\alpha[I_\varepsilon(t)^2]\to0$ for every fixed $t$,
	and therefore $I_\varepsilon(t)\to0$ in $\mathbf P_\alpha$-probability.
	Because this holds for every realisation of $\tau$ (i.e. for every $\omega_X$), integrating over $\tau$ gives $\mathbf E_{\alpha,X}[I_\varepsilon(t)^2]\to0$, and therefore $J_\varepsilon(t)\to\bar\alpha L_t$ in $\mathbf P_{\alpha,X}$-probability.
	
	For any partition $0=t_0<t_1<\dots<t_m=T$ and any $t\in[t_{k-1},t_k]$,
	the monotonicity implies
	\[
	J_\varepsilon(t_{k-1})-\bar\alpha L_{t_k}\le J_\varepsilon(t)-\bar\alpha L_t\le
	J_\varepsilon(t_k)-\bar\alpha L_{t_{k-1}}.
	\]
	By writing \(J_\varepsilon(t_{k-1})-\bar\alpha L_{t_k} = (J_\varepsilon(t_{k-1})-\bar\alpha L_{t_{k-1}}) - (\bar\alpha L_{t_{k}}-\bar\alpha L_{t_{k-1}}) \) and \(J_\varepsilon(t_{k})-\bar\alpha L_{t_{k-1}} = (J_\varepsilon(t_{k-1})-\bar\alpha L_{t_{k}}) + (\bar\alpha L_{t_{k}}-\bar\alpha L_{t_{k-1}}) \), we get
	\[
	|J_\varepsilon(t)-\bar\alpha L_t|\le \max_{i=k-1,k}|J_\varepsilon(t_i)-\bar\alpha L_{t_i}|
	+\bar\alpha(L_{t_k}-L_{t_{k-1}}).
	\]
	Taking the supremum over $t$ gives
	\[
	\mathcal E_\varepsilon \le \max_{k=0,\dots,m}|J_\varepsilon(t_k)-\bar\alpha L_{t_k}|
	+ \bar\alpha\max_{k=1,\dots,m}(L_{t_k}-L_{t_{k-1}}).
	\]
	
	Now $L_t$ is almost surely continuous on $[0,T]$.
	For any $\delta>0$ and $\eta>0$, choose a partition fine enough so that
	\[
	\mathbf P_X\Bigl(\bar\alpha\max_k(L_{t_k}-L_{t_{k-1}})>\frac{\delta}{2}\Bigr)<\frac{\eta}{2}.
	\]
	For this fixed partition, the finite set $\{t_k\}$ has the property that
	\[
	\max_k|J_\varepsilon(t_k)-\bar\alpha L_{t_k}|\to0\quad\text{in }\mathbf P_{\alpha,X}\text{-probability},
	\]
	because each term converges and the maximum of finitely many convergent sequences converges.
	Thus, for $\varepsilon$ sufficiently small,
	\[
	\mathbf P_{\alpha,X}\Bigl(\max_k|J_\varepsilon(t_k)-\bar\alpha L_{t_k}|>\frac{\delta}{2}\Bigr)<\frac{\eta}{2}.
	\]
	Combining the two estimates,
	\[
	\mathbf P_{\alpha,X}\bigl(\mathcal E_\varepsilon>\delta\bigr)<\eta
	\]
	for all small $\varepsilon$. Hence $\mathcal E_\varepsilon\to0$ in $\mathbf P_{\alpha,X}$-probability.
	
	Recall that $\mathcal E_\varepsilon \le (\|S\|_\infty+\bar\alpha)L_T$. For reflected Brownian motion in a bounded $C^2$ domain, $\mathbf E[L_T]<\infty$ (since it is finite for every finite time \(T\)). Hence the family $\{\mathcal E_\varepsilon\}$ is uniformly integrable. Together with $\mathcal E_\varepsilon\to0$ in $\mathbf P_{\alpha,X}$-probability, this implies $\mathbf E_{\alpha,X}[\mathcal E_\varepsilon]\to0$ (Vitali's convergence theorem).
	By Tonelli's theorem,
	$\mathbf E_{\alpha,X}[\mathcal E_\varepsilon] = \mathbf E_\alpha\bigl[\mathbf E_x[\mathcal E_\varepsilon]\bigr]$.
	The non‑negative random variable $\mathbf E_x[\mathcal E_\varepsilon]$ (which depends only on $\alpha$)
	has expectation tending to zero; therefore it converges to zero in $\mathbf P_\alpha$-probability.
	Consequently,
	\[
	\sup_{t\in[0,T]}|u^\varepsilon(t,x)-u^0(t,x)|\le \|f\|_\infty\,\mathbf E_x[\mathcal E_\varepsilon]
	\longrightarrow 0\quad\text{in }\mathbf P_\alpha\text{-probability},
	\]
	which completes the proof.
\end{proof}
	Note that, consistently with the notation introduced in Theorem \ref{thm:quenched}, the function $u^\varepsilon(t,x)$ corresponds to the action of the evolution family $(S_{0,t}^{\alpha^\varepsilon})_{t \ge 0}$ on the initial datum $f$, i.e., $u^\varepsilon(t,x) = (S_{0,t}^{\alpha^\varepsilon} f)(x)$.
	
	\begin{remark}
		We use the term ``fast switching'' to describe the asymptotic regime in which the random boundary parameter evolves on a time scale that is much faster than the natural time scale of the reflected diffusion. Concretely, one rescales only the environment (the process \(\alpha\)) so that its fluctuations become rapid compared with the motion of \(X\); the particle path \(X\) is observed at its original physical time scale. 
		
		This choice reflects a clear modelling assumption: the medium (or the collection of microscopic gates on the boundary) changes state very frequently, while the diffusing particle experiences these rapid changes as an averaged, effective boundary reactivity. Under stationarity and ergodicity of the fast process, this intuitive picture is made rigorous by averaging arguments that replace the pathwise time-dependent reactivity by its mean in the limit of infinitely fast switching.
	\end{remark}

\section{Application to Gated Boundary Reactions}

We illustrate the practical relevance of our main results by considering a model for the binding of diffusive ligands to receptors located on a cell membrane surface. This is a fundamental problem in cellular signaling and biophysics (see \cite{BergPurcell1977,Zwanzig1990}).

In realistic biological environments, the receptors on the boundary are not permanently active. Their state fluctuates stochastically due to thermal conformational changes, a phenomenon known as \textit{gating}. We model the specific reactivity of the boundary at time $t$ by the process $\{\alpha_t\}_{t\ge 0}$.

We assume the receptors switch between two possible states: a \textit{closed state} (inert) and an \textit{open state} (reactive).
\begin{itemize}
	\item \textit{Closed State ($s=0$):} The receptor is sterically inaccessible. The boundary acts as a purely reflecting barrier. In this state, we have $\alpha_t = 0$.
	\item \textit{Open State ($s=\kappa$):} The receptor is exposed and can bind the ligand. The boundary is partially absorbing with an intrinsic reaction rate $\kappa > 0$. In this state, we have $\alpha_t = \kappa$.
\end{itemize}

Mathematically, the process $\{\alpha_t\}_{t\ge 0}$ is a continuous-time Markov chain with state space $S=\{0, \kappa\}$. The transitions are governed by the switching rates $\lambda_{\text{on}}$ (from closed to open) and $\lambda_{\text{off}}$ (from open to closed). The infinitesimal generator $Q$ of the chain is given by:
\begin{equation}\label{eq:generator_Q}
	Q = \begin{pmatrix} 
		-\lambda_{\text{on}} & \lambda_{\text{on}} \\ 
		\lambda_{\text{off}} & -\lambda_{\text{off}} 
	\end{pmatrix}.
\end{equation}
As in \cite{Szabo1982}, the concentration of ligands, denoted by $u(t,x)$, evolves according to the diffusion equation. In this paper, the interaction with gating receptors is modeled through a stochastic Robin boundary condition for the heat equation, where the boundary reactivity undergoes Markovian switching in time:
\begin{equation}\label{eq:gated_system}
	\begin{cases}
		\partial_t u(t,x) = \tfrac{1}{2}\Delta u(t,x), & x \in D, \, t > 0, \\
		\partial_n u(t,x) + \alpha_t u(t,x) = 0, & x \in \partial D, \, t > 0,
	\end{cases}
\end{equation}
where we assumed a unitary diffusion coefficient for simplicity (or absorbed it into the time scale). The term $\alpha_t u$ represents the time-dependent flux of reaction at the membrane.

\subsection*{The Fast Switching Limit}
In many biological regimes, the conformational changes of the receptors occur on the nanosecond to microsecond scale, whereas the typical diffusion time of the ligand across the cell is on the order of milliseconds or seconds. This separation of time scales implies that the switching of $\alpha_t$ is extremely fast compared to the motion of the Brownian particle.

Let $\varepsilon > 0$ represent the ratio between the characteristic time of the gating process and the diffusion time. We consider the rescaled process $\alpha^\varepsilon_t = \alpha_{t/\varepsilon}$. Since the Markov chain on $S=\{0, \kappa\}$ is irreducible and finite, it possesses a unique invariant measure $\pi = (\pi_0, \pi_\kappa)$, and the process is exponentially mixing. Theorem \ref{thm:averaging} (applied to the stationary process $\alpha_t$) guarantees that, in the limit $\varepsilon \to 0$, the effective boundary condition is a Robin condition with constant coefficient $\bar\alpha$ given by the stationary mean of $\alpha_t$; moreover, due to exponential mixing, the same limit holds for any initial distribution of the gating states.

To compute the effective parameter $\bar{\alpha}$, we must determine the stationary distribution $\pi = (\pi_0, \pi_\kappa)$ of the Markov chain. The invariant measure satisfies the equation $\pi Q = 0$ subject to the normalization condition $\pi_0 + \pi_\kappa = 1$. Using the generator in \eqref{eq:generator_Q}, the system becomes:
\[
\begin{cases}
	-\lambda_{\text{on}} \pi_0 + \lambda_{\text{off}} \pi_\kappa = 0, \\
	\pi_0 + \pi_\kappa = 1.
\end{cases}
\]
From the first equation, we obtain $\pi_\kappa = \frac{\lambda_{\text{on}}}{\lambda_{\text{off}}} \pi_0$. Substituting this into the normalization condition:
\[
\pi_0 \left( 1 + \frac{\lambda_{\text{on}}}{\lambda_{\text{off}}} \right) = 1 
\implies \pi_0 = \frac{\lambda_{\text{off}}}{\lambda_{\text{on}} + \lambda_{\text{off}}}.
\]
Consequently, the stationary probability of being in the open (reactive) state is:
\[
\pi_\kappa = 1 - \pi_0 = \frac{\lambda_{\text{on}}}{\lambda_{\text{on}} + \lambda_{\text{off}}}.
\]
The effective reactivity $\bar{\alpha}$ is the average of the process with respect to this stationary measure:
\[
\bar{\alpha} = \mathbf{E}[\alpha_\infty] = 0 \cdot \pi_0 + \kappa \cdot \pi_\kappa = \kappa \frac{\lambda_{\text{on}}}{\lambda_{\text{on}} + \lambda_{\text{off}}}.
\]
Therefore, in the limit of fast gating, the effective boundary condition is given by:
\begin{equation}
	\partial_n u + \bar{\alpha} u = 0.
\end{equation}
This result provides a rigorous mathematical justification for the standard biochemical approximation \cite{Zwanzig1990}, which replaces the rapidly fluctuating receptors with a uniformly permeable membrane having an effective rate constant $\bar{\alpha}$.

	\section*{Acknowledgements}
	The author thanks to the group INdAM-GNAMPA for the support under their Grants.\\
	%The research has been mostly funded by MUR under the project PRIN 2022 - 2022XZSAFN - CUP B53D23009540006 - PNRR M4.C2.1.1.: Anomalous Phenomena on Regular and Irregular Domains:
	%Approximating Complexity for the Applied Sciences. \\
	%Web Site: \url{https://www.sbai.uniroma1.it/~mirko.dovidio/prinSite/index.html}.


\begin{thebibliography}{99}
		
		\bibitem{Arendt2018}
		W.~Arendt, S.~Kunkel and M.~Kunze, ``Diffusion with nonlocal Robin boundary conditions,'' \textit{J. Math. Soc. Japan}, vol.~70, no.~4, pp.~1523--1556, 2018.
		
		\bibitem{ArendtMonniaux2016}
		W.~Arendt and S.~Monniaux, ``Maximal regularity for non-autonomous Robin boundary conditions,'' \textit{Math. Nachr.}, vol.~289, no.~11-12, pp.~1325--1340, 2016.
		
		\bibitem{Berezhkovskii2004}
		A.~M.~Berezhkovskii, Y.~A.~Makhnovskii, M.~I.~Monine, V.~Y.~Zitserman and S.~Y.~Shvartsman, ``Boundary homogenization for trapping by patchy surfaces,'' \textit{J. Chem. Phys.}, vol.~121, no.~22, pp.~11390--11397, 2004.
		
		\bibitem{Berezhkovskii2006}
		A.~M.~Berezhkovskii, M.~I.~Monine, C.~B.~Muratov and S.~Y.~Shvartsman, ``Homogenization of boundary conditions for surfaces with regular arrays of traps,'' \textit{J. Chem. Phys.}, vol.~124, 044508, 2006.
		
		\bibitem{BergPurcell1977}
		H.~C.~Berg and E.~M.~Purcell, ``Physics of chemoreception,'' \textit{Biophys. J.}, vol.~20, pp.~193--219, 1977.
		
%		\bibitem{BlumenthalGetoor2007}
%		R.~M.~Blumenthal and R.~K.~Getoor, \textit{Markov Processes and Potential Theory}, Courier Corporation, 2007.
		
		\bibitem{Bressloff2015}
		P.~C.~Bressloff and S.~D.~Lawley, ``Stochastically gated diffusion-limited reactions for a small target in a bounded domain,'' \textit{Phys. Rev. E}, vol.~92, no.~6, p.~062117, 2015.
		
		\bibitem{BressloffLawley2015}
		P.~C.~Bressloff and S.~D.~Lawley, ``Escape from a potential well with a randomly switching boundary,'' \textit{J. Phys. A: Math. Theor.}, vol.~48, 225001, 2015.
		
%		\bibitem{Cerrai2009}
%		S.~Cerrai, ``A Khasminskii type averaging principle for stochastic reaction-diffusion equations,'' \textit{Ann. Appl. Probab.}, vol.~19, no.~3, pp.~899--948, 2009.
		
		\bibitem{ColantoniDovidio2025}
		F.~Colantoni and M.~D'Ovidio, ``Elastic Brownian motion with random jumps from the boundary,'' \textit{arXiv preprint arXiv:2511.01455}, 2025.
		
		\bibitem{ColantoniDOvidioPagnini2025}
		F.~Colantoni, M.~D'Ovidio, and G.~Pagnini, ``Time reversal of Reflected Brownian Motion with Poissonian Resetting,'' \textit{J. Stat. Phys.}, vol.~192, p.~147, 2025.
		
		\bibitem{DovidioFCAA}
		M.~D'Ovidio, ``Fractional boundary value problems and elastic sticky Brownian motions,'' \textit{Fract. Calc. Appl. Anal.}, vol.~27, no.~5, pp.~2162--2202, 2024.
		
		\bibitem{Dovidio2024}
		M.~D'Ovidio, ``Fractional Boundary Value Problems and elastic sticky Brownian motions, II: The bounded domain,'' \textit{arXiv preprint arXiv:2205.04162}, 2024.
		
		\bibitem{FiloLuckhaus1995}
		J.~Filo and S.~Luckhaus, ``Asymptotic expansion for a periodic boundary condition,'' \textit{J. Differential Equations}, vol.~120, pp.~133--179, 1995.
		
		\bibitem{FiloLuckhaus2000}
		J.~Filo and S.~Luckhaus, ``Homogenization of a boundary condition for the heat equation,'' \textit{J. Eur. Math. Soc.}, vol.~2, pp.~217--258, 2000.
		
		\bibitem{Friedman1995}
		A.~Friedman, C.~Huang and J.~Yong, ``Effective permeability of the boundary of a domain,'' \textit{Comm. Partial Differential Equations}, vol.~20, pp.~1235--1257, 1995.
		
		\bibitem{Grebenkov2019}
		D.~S.~Grebenkov, ``Probability distribution of the boundary local time of reflected Brownian motion in Euclidean domains,'' \textit{Phys. Rev. E}, vol.~100, 062110, 2019.
		
		\bibitem{Grebenkov2024review}
		D.~S.~Grebenkov, ``Encounter-based approach to target search problems: a review,'' in \textit{Target Search Problems}, Springer, pp. 77-105, 2024.
		
		\bibitem{ItoMcKean1974}
		K.~It\^{o} and H.~P.~McKean, Jr., \textit{Diffusion Processes and Their Sample Paths}, Springer, 1974.
		
%		\bibitem{Khasminskii1968}
%		R.~Z.~Khasminskii, ``On the principle of averaging the It\^{o} stochastic differential equations,'' \textit{Kybernetika}, vol.~4, pp.~260--279, 1968.
		
		\bibitem{Lawley2016}
		S.~D.~Lawley, ``Boundary value problems for statistics of diffusion in a randomly switching environment: PDE and SDE perspectives,'' \textit{SIAM J. Appl. Dyn. Syst.}, vol.~15, no.~3, pp.~1410--1433, 2016.
		
		\bibitem{LawleyKeener2015}
		S.~D.~Lawley and J.~P.~Keener, ``A new derivation of Robin boundary conditions through homogenization of a stochastically switching boundary,'' \textit{SIAM J. Appl. Dyn. Syst.}, vol.~14, no.~4, pp.~1845--1867, 2015.
		
		\bibitem{LawleyMattinglyReed2015}
		S.~D.~Lawley, J.~C.~Mattingly and M.~C.~Reed, ``Stochastic switching in infinite dimensions with applications to random parabolic PDE,'' \textit{SIAM J. Math. Anal.}, vol.~47, no.~6, pp.~3035--3063, 2015.
		
		\bibitem{LionsSznitman1984}
		P.-L.~Lions and A.-S.~Sznitman, ``Stochastic differential equations with reflecting boundary conditions,'' \textit{Commun. Pure Appl. Math.}, vol.~37, pp.~511--537, 1984.
		
		\bibitem{Mao2000}
		X.~Mao, A.~Matasov and A.~B.~Piunovskiy, ``Stochastic differential delay equations with Markovian switching,'' \textit{Bernoulli}, vol.~6, no.~1, pp.~73--90, 2000.
		
		\bibitem{Mao2006}
		X.~Mao and C.~Yuan, \textit{Stochastic Differential Equations with Markovian Switching}, Imperial College Press, 2006.
		
		\bibitem{Ocejo2020}
		A.~Ocejo, ``Integral equation characterization of the Feynman–Kac formula for a regime-switching diffusion,'' \textit{Results Appl. Math.}, vol.~5, p.~100087, 2020.
		
		\bibitem{Papanicolaou1990}
		V.~G.~Papanicolaou, ``The probabilistic solution of the third boundary value problem for second order elliptic equations,'' \textit{Probab. Theory Relat. Fields}, vol.~87, pp.~27--77, 1990..
		
		\bibitem{Singer2008}
		A.~Singer, Z.~Schuss, C.~Lumsden, A.~Osipov and D.~Holcman, ``Partially reflected diffusion,'' \textit{SIAM J. Appl. Math.}, vol.~68, no.~3, pp.~844--868, 2008.
		
		\bibitem{Szabo1982}
		A.~Szabo, D.~Shoup, S.~H.~Northrup, and J.~A.~McCammon, ``Stochastically gated diffusion-influenced reactions,'' \textit{J. Chem. Phys.}, vol.~77, no.~9, pp.~4484--4493, 1982.
		
		\bibitem{Taira2004}
		K.~Taira, \textit{Semigroups, Boundary Value Problems and Markov Processes}, Springer-Verlag, Berlin/Heidelberg, 2004.
		
		\bibitem{WeiWangNane2025}
		Z.~Wei, Y.~Wang, and E.~Nane, ``Feynman–Kac formula for regime-switching general diffusions,'' \textit{Appl. Math. Lett.}, vol.~168, p.~109573, 2025.
		
		\bibitem{ZhuYinBaran2015}
		C.~Zhu, G.~Yin, and N.~A.~Baran, ``Feynman–Kac formulas for regime-switching jump diffusions and their applications,'' \textit{Stochastics}, vol.~87, no.~6, pp.~1000--1032, 2015.
		
		\bibitem{Zwanzig1990}
		R.~Zwanzig, ``Rate processes with dynamical disorder,'' \textit{Acc. Chem. Res.}, vol.~23, no.~5, pp.~148--152, 1990.
		
	\end{thebibliography}
\end{document}